\documentclass[10pt]{amsart} 
\usepackage{amsmath,amssymb, amsthm,amsbsy} 
\usepackage[dvips]{graphicx}
\usepackage[usenames,dvipsnames]{pstricks} 
\usepackage{epsfig}
\usepackage{pst-grad} 
\usepackage{pst-plot} 
\usepackage[colorlinks=true, linkcolor=magenta, citecolor=cyan, urlcolor=blue,hyperfootnotes=false]{hyperref}

\newcommand{\R}{{\mathbb R}}
 
\newcommand{\Z}{{\mathbb Z}}
\newcommand{\C}{{\mathbb C}}

 \renewcommand{\geq }{\geqslant}
 \renewcommand{\leq }{\leqslant}

 \usepackage{mathtools}
\DeclarePairedDelimiter{\abs}{\lvert}{\rvert}

\DeclarePairedDelimiter{\norma}{\lVert}{\rVert}

\usepackage{braket}
\usepackage{color} 
\newcommand{\Rn}{{\mathbb R^n}}
\newcommand{\ignora}[1]{}
{\left\lbrace\begin{array}{@{}l@{}}}%
{\end{array}\right.}  

\numberwithin{equation}{section}
\newtheorem{theorem}{Theorem}[section]
\newtheorem{corollary}[theorem]{Corollary}
\newtheorem{lemma}[theorem]{Lemma}
\newtheorem{proposition}[theorem]{Proposition}
\theoremstyle{definition} 

\newtheorem{example}[theorem]{Example}

\newtheorem{remark}[theorem]{Remark}

\begin{document}

\title[Harmonic Oscillators and related models: Gaussian decay]{Gaussian decay of Harmonic Oscillators and related models}

\author{B.~Cassano}
\address{Biagio Cassano: 
BCAM - Basque Center for Applied Mathematics, 
Alameda Mazarredo 14, 48009 Bilbao, Basque Country - Spain.}
\email{bcassano@bcamath.org}

\author{L.~Fanelli}
\address{Luca Fanelli: SAPIENZA Universit$\grave{\text{a}}$ di Roma, Dipartimento di Matematica, P.le A. Moro 5, 00185 Roma, Italy.}
\email{fanelli@mat.uniroma1.it}


\thanks{
The two authors were supported by the Italian project FIRB 2012 {\it Dispersive Dynamics: Fourier Analysis and Calculus of Variations}.
The first author was also supported by ERCEA Advanced Grant 2014 669689 - HADE, by the MINECO project MTM2014-53850-P, by Basque Government project IT-641-13 and also by the Basque Government through the BERC 2014-2017 program and by Spanish Ministry of Economy and Competitiveness MINECO: BCAM Severo Ochoa excellence accreditation SEV-2013-0323.}

\begin{abstract}
We prove that the decay of the eigenfunctions of harmonic oscillators, uniform electric or magnetic fields is not stable under 0-order complex perturbations, even if bounded, of these Hamiltonians, in the sense that we can produce solutions to the evolutionary Schr\"odinger flows associated to the Hamiltonians, with a stronger Gaussian decay at two distinct times. We then characterize, in a quantitative way, the sharpest possible Gaussian decay of solutions as a function of the oscillation frequency or the strength of the field, depending on the Hamiltonian which is considered. This is connected to the Hardy's Uncertainty Principle for free Schr\"odinger evolutions.
\end{abstract}

\date{\today}
\maketitle

\section{Introduction}

Let us consider an electromagnetic  Schr\"odinger Hamiltonian of the form
\begin{equation*}
  H=-\Delta_A+V(x),
\end{equation*}
where $\Delta_A:=(\nabla-iA)^2$ and the potentials $A,V$ are given by
\begin{equation*}
  A:\R^n\to\R^n,
  \qquad
  V:\R^n\to\R.
\end{equation*}
We assume that $H$ can be defined as a self-adjoint operator on a suitable subset $X\subset L^2(\R^n)$, so that the Schr\"odinger flow $e^{itH}$ is well-defined by functional calculus. Moreover, we assume that $H$ has pure point spectrum, and its eigenvalues form an orthonormal basis of $L^2(\R^n)$. This is a typical situation, if unbounded (at infinity) perturbations are involved, like harmonic oscillators or uniform electric or magnetic fields, as we see in the sequel. In this framework, we have a countable set of standing-waves of the form $e^{itH}\psi_k=e^{i\lambda_kt}\psi_k$, being $\lambda_k$ an eigenvalue and $\psi_k$ a corresponding eigenfunction of $H$. The space-decay at infinity of these objects is, independently of time, the one of the eigenfunctions $\psi_k$, which is, in most cases, exponential. The two most relevant models are the following ones.

\begin{example}[Quantum harmonic oscillator]
Consider the 1D-equation
\begin{equation}\label{eq:QHOexample}
  i\partial_t u - \partial_{xx} u + \frac{\omega^2 x^2}{4} u =0.
\end{equation}

It is well known (see e.g. Chapter 7 in \cite{thaller}) 
that the Hamiltonian
\begin{equation}\label{eq:defnHomega}
  H_\omega u = - \partial_{xx} u + \frac{\omega^2 x^2}{4} u 
\end{equation}
has a pure point spectrum 
\begin{equation*}
\sigma(H_\omega)=\sigma_{pp}(H_\omega)=
\left\{E_m = \omega\left(m + \frac12\right) \colon m=0,1,\dots\right\}
\end{equation*}
with eigenfunctions
\begin{equation*}
    \psi_m(x) = 
    h_m \,
    H_m\left(\sqrt{\frac{\omega}{2}} x \right)
    \,
    e^{- \frac{\omega x^2}{4}} , \qquad m = 0,1,2,\ldots,
\end{equation*}
where $h_m>0$ are normalization constants and the functions $H_m$ are the \emph{Hermite polynomials}
\begin{equation*}\label{eq:defnHermite}
 	H_m(x)=(-1)^m e^{x^2}\frac{d^m}{dx^m}\left(e^{-x^2}\right).
\end{equation*}
\end{example}
\begin{example}[Uniform magnetic field]\label{ex:UM}
Consider the 2D-equation
\begin{equation}\label{eq:USchrodingereq}
  i \partial_t u - \left(\nabla - i A \right)^2 u =0   , \quad
  A(x_1,x_2)=\frac{b}{2}(-x_2,x_1), \quad b\neq 0.
\end{equation}
One can easily see (see e.g. Chapter 8 of \cite{thaller}) that the Hamiltonian 
$H(A)=- \left(\nabla - i A \right)^2$ 
has pure point spectrum
\begin{equation*}
  \sigma(H(A))=\sigma_{pp}(H(A))=\left\{F_k=\abs{b}\left(k+\frac12\right)\colon
    k=0,1,\dots\right\}.
\end{equation*}
For $k=m+(\abs{l}-(\text{sgn} \,b )l)/2$, the eigenvalues 
(\emph{Landau Levels}) are associated to the eigenfunctions
\begin{equation}\label{eq:autofunzioniUM}
  \phi_{m,l}(x)=
  p_{m,l} r^{\abs{l}} \binom{m+\abs{l}}{m}^{-1}
          L^{(\abs{l})}_m\left(\frac{\abs{b}}{2}r^2\right)
          e^{il\varphi}
          e^{-\frac{\abs{b}r^2}{4}}, \quad l \in \Z, m=0,1,\dots,
\end{equation}
with $x=r\varphi\in\R^n$, for $r>0$, $\varphi \in \mathbf S^{n-1}$;  $p_{m,l}>0$ are 
normalization constants and $L^{(\alpha)}_m$ are the 
\emph{generalized Laguerre Polynomials}
\begin{equation*}
L_m^{(\alpha)}(x) = {x^{-\alpha} e^x \over m!}{d^m \over dx^m} \left(e^{-x} x^{m+\alpha}\right).
\end{equation*}
\end{example}
In both the previous cases, the eigenfunctions have $L^2$--gaussian decay, namely
\begin{equation*}
  \norma*{ e^{|\cdot|^2/\alpha^2} 
  \psi_n(\cdot)}_{L^2} 
  +
  \norma*{ e^{|\cdot|^2/\beta^2} 
  \phi_{n,l}(\cdot)}_{L^2}<+\infty,
  \qquad l \in \Z, m = 0,1,2,\ldots,
\end{equation*}
for all $\alpha>2/\sqrt{\omega}$ and $\beta>2/\sqrt{b}$ respectively. Therefore, the corresponding standing-wave solutions
\begin{equation*}
  u(x,t)=e^{i E_m t}\psi_m(x), \quad
  v(x,t)=e^{i F_m t}\phi_{m,l}(x),
\end{equation*}
to \eqref{eq:QHOexample}
and \eqref{eq:USchrodingereq}, respectively, satisfy
\begin{equation*}
  \begin{split}
  &\norma*{ 
     e^{|\cdot|^2/\alpha^2} 
    u(\cdot,0)
    }_{L^2}
+
  \norma*{ 
    e^{|\cdot|^2/\alpha^2} 
    u(\cdot,1)
    }_{L^2}
    <+\infty, 
\\
  &\norma*{ 
    e^{|\cdot|^2/\beta^2} 
    v(\cdot,0)
}_{L^2}
+
  \norma*{ 
    e^{|\cdot|^2/\beta^2} 
    v(\cdot,1)
}_{L^2}
<+\infty,
\end{split}
\end{equation*}
provided
\begin{equation}\label{eq:condsbag}
\alpha^2 > \frac4\omega,
\qquad
 \beta^2 > \frac4b.
\end{equation}

The first question of this manuscript is whether it is possible to obtain solutions with stronger gaussian decay at two distinct times than the one given by \eqref{eq:condsbag}, by perturbing \eqref{eq:QHOexample}
and \eqref{eq:USchrodingereq} with a zero-order term $V(x,t)\in L^\infty(\R^n\times\R)$. The answer is, not surprisingly, positive if we allow complex perturbations, as we show in the following theorems.

\begin{theorem}\label{thm:harmex}
\label{prop:QHOcounterexample}
Define the following function
\begin{equation}\label{eq:QHOexampleSharp}
  \begin{split}
    u(x,t)=&(\cos{\omega t})^{-\frac{n}{2}}
           (1+ih \tan{\omega t})^{2k-\frac{n}{2}}
           \left(1+\frac{h\omega \abs{x}^2}{\cos^2{\omega t}}\right)^{-k}\\
           &\times\exp\left[-\frac{h\omega\abs{x}^2}{4(\cos^2{\omega t}+h^2\sin^2{\omega t})}
                     + i \tan{\omega t}
                         \left(1-h^2\frac{1+\tan^2{\omega t}}{1+h^2\tan^2{\omega t}}\right)\right]
           \end{split}
\end{equation}
with $k > n/2$ and
\begin{equation*}
  h=\frac{2\pm\sqrt3}{\tan{\omega/2}}.
\end{equation*}
Then $u\in \mathcal C\left(\left[-\frac{1}{2},\frac{1}{2}\right];L^2(\R^n)\right)$ 
and it is solution to
\begin{equation*}
  i \partial_t u -\Delta u + V(x,t) u + \frac{\omega^2\abs{x}^2}{4}u=0,
\end{equation*}
with 
\begin{equation}\label{eq:definizioneVesempio}
  V(x,t)=\frac{2k}{\frac{\cos^2{\omega t}}{h\omega}+\abs{x}^2}
         \left[
           \frac{1}{1+ih\tan{\omega t}}+n
           -\frac{2(1+k)\abs{x}^2}{\frac{\cos^2{\omega t}}{h\omega}+\abs{x}^2}
         \right].
\end{equation}
The function $u$ satisfies
\begin{equation*}
  \norma*{e^{\abs{\cdot}^2/\tilde \alpha^2}u\left(\cdot,-\frac{1}{2}\right)}_{L^2}=
  \norma*{e^{\abs{\cdot}^2/\tilde \alpha^2}u\left(\cdot,\frac{1}{2}\right)}_{L^2} 
  < +\infty
\end{equation*}
for some $\widetilde\alpha \in \R$ such that $\widetilde\alpha^2=\frac{4\sin{\omega}}{\omega}$.
\end{theorem}
The analogous result for \eqref{eq:USchrodingereq} is as follows.
\begin{theorem}\label{prop:UMcounterexample}
For $\omega:=b$, $n=2$, $k>4$, the function $u$ in \eqref{eq:QHOexampleSharp} is solution to
\begin{equation*}
  i \partial_t u - (\nabla - i A(x))^2 u + V(x,t) u = 0,
\end{equation*}
with $V$ defined in \eqref{eq:definizioneVesempio}
and we have that 
\begin{equation*}
  \norma*{e^{\abs{\cdot}^2/\tilde \alpha^2}u\left(\cdot,-\frac{1}{2}\right)}_{L^2}=
  \norma*{e^{\abs{\cdot}^2/\tilde \alpha^2}u\left(\cdot,\frac{1}{2}\right)}_{L^2} 
  < +\infty
\end{equation*}
for some $\tilde\alpha \in \R$ such that $\widetilde\alpha^2=\frac{4\sin{b}}{b}$.
\end{theorem}

\begin{remark}
Notice that the conditions $\widetilde\alpha^2=\frac{4\sin{\omega}}{\omega}$ and $\widetilde\alpha^2=\frac{4\sin{b}}{b}$ in the previous results provide a stronger decay than the one in \eqref{eq:condsbag}.
The potential $V$ in \eqref{eq:definizioneVesempio} is complex-valued, and this is quite likely necessary in order to produce examples with such a decay property. Indeed, this kind of question has a stationary counterpart with a well known manifestation in the examples by Meshkov \cite{M} and Cruz-Sanpedro \cite{C}.
Also in that case, the complex nature of the perturbation is essential, as it has been recently proved, at least in part, by Kenig, Silvestre
and Wang in \cite{KSW} 
and extended to more general elliptic operators by Davey, Kenig and Wang in \cite{DKW}.
\end{remark}

\begin{remark}\label{rem:conne}
Equations \eqref{eq:QHOexample} and \eqref{eq:USchrodingereq} 
  are strictly connected to each other: indeed observe that
  \begin{equation*}
    H(A)= - (\nabla - i A)^2 
        = - \Delta  + \left(\frac{b}2 \right)^2\abs{x}^2
          + i \frac{b}2 (-x_2,x_1)\cdot \nabla 
        = H_{b} - \frac{b}2 L
  \end{equation*}
  where $H_{b}$ is defined in \eqref{eq:defnHomega} and $L:=-i(x_1,x_2)\wedge\nabla =-i(-x_2,x_1)\cdot \nabla$ is
  the \emph{angular momentum operator}. Since $[H_{b},L]=0$, we have
  \begin{equation*}
    e^{iH(A)t}=e^{iH_{b}t} e^{-i\frac{b}2 L t}
  \end{equation*}
  and we notice that $L$ is the generator of rotations: for $\varphi >0$ 
  \begin{equation*}
    e^{-iL\varphi}\psi(x)=\psi(R(-\varphi)x)
  \end{equation*}
  where $R(\varphi)$ is the counterclockwise rotation in $\R^2$ of an angle $\varphi$:
  \begin{equation*}
    R(\varphi)=
    \begin{bmatrix}
      \cos\varphi & -\sin\varphi \\
      \sin\varphi & \cos\varphi
    \end{bmatrix}.
  \end{equation*}
  Therefore the evolution according to equation \eqref{eq:USchrodingereq} 
  is the composition between a rotation and a harmonic oscillator flow.
  Finally notice that, since $H_{b}$ and $L$ commute, one can find a common 
  base of eigenfunctions. Indeed $\phi_{m,l}$ in \eqref{eq:autofunzioniUM}
  verify $H_{b} \phi_{m,l} = \abs{b}\left(m + \abs{l}/2 + 1/2\right) \phi_{m,l}$ 
  and $L \phi_{m,l} = l \phi_{m,l}$, for $m=0,1,\dots$ and $l \in \Z$.
\end{remark}
Our next step is to prove that Theorems \ref{thm:harmex} and \ref{prop:UMcounterexample} are sharp: namely, that non-trivial solutions with stronger Gaussian decay than the one in \eqref{eq:QHOexampleSharp}, at two distinct times, cannot exist. This follows a program which has been developed in the magnetic free case $A\equiv 0$ without harmonic oscillators
by Escauriaza, Kenig, Ponce, and Vega
in the sequel of papers \cite{EKPV0,EKPV1,EKPV2,EKPV3,EKPV4},
and  with Cowling in \cite{CEKPV}
and continued in the magnetic case $A \not\equiv 0$ in \cite{BFGRV} and in \cite{CF}.
In all these references, the motivation is given by the Hardy Uncertainty Principle and its connection with Schr\"odinger evolutions:

{\it if $f(x)=O\left(e^{-|x|^2/\beta^2}\right)$ and its Fourier transform
$\hat f(\xi)=O\left(e^{-4|\xi|^2/\alpha^2}\right)$, then}
\begin{align*}
  \alpha\beta<4
  &
  \Rightarrow
  f\equiv0
  \\
  \alpha\beta=4
  &
  \Rightarrow
  f\ is\ a \ constant\ multiple\ of\ e^{-\frac{|x|^2}{\beta^2}}.
\end{align*}
The evolutionary version of the previous is the following:

{\it if $u(x,0)=O\left(e^{-|x|^2/\beta^2}\right)$ and $u(x,T):=e^{iT\Delta}u(x,0)=
O\left(e^{-|x|^2/\alpha^2}\right)$, then}
\begin{align*}
  \alpha\beta<4T
  &
  \Rightarrow
  u\equiv0
  \\
  \alpha\beta=4T
  &
  \Rightarrow
  u(x,0)\ is\ a \ constant\ multiple\ of\ e^{-\left(\frac{1}{\beta^2}+\frac{i}{4T}\right)|x|^2}.
\end{align*}
An $L^2$-versions of the previous results (see \cite{SST}) is also available:
\begin{align*}
  e^{|x|^2/\beta^2}f \in L^2,\  e^{4|\xi|^2/\alpha^2}\hat f \in L^2,\
  \alpha\beta\leq4
  &
  \Rightarrow
  f\equiv0
  \\
  e^{|x|^2/\beta^2}u(x,0)\in L^2,\  e^{|x|^2/\alpha^2}e^{iT\Delta}u(x,0) \in L^2,\
  \alpha\beta\leq4T
  &
  \Rightarrow
  u\equiv0.
\end{align*}
We mention \cite{BD, FS, SS} as interesting surveys about this topic.
In \cite{CEKPV,EKPV0,EKPV1,EKPV2,EKPV3,EKPV4}, the authors investigated the validity of the previous statements for zero-order perturbations of the Schr\"odinger equation of the form 
\begin{equation*}\label{3eq:main2}
  \partial_t u = i (\Delta + V(t,x))u.
\end{equation*}
We can briefly summarize their strongest results: if $V(t,x)\in L^\infty$ is the sum of a real-valued potential $V_1$ and a sufficiently decaying complex-valued potential $V_2$, and $\norma{e^{\abs{x}^2/\beta^2}u(0)}_{L^2}
+\norma{e^{\abs{x}^2/\alpha^2}u(T)}_{L^2}<+\infty$, with $\alpha\beta<4T$, then $u\equiv0$. Moreover, the result is sharp: indeed, Theorem 2 in \cite{EKPV3} provides an example of a (complex) potential $V$ for which there exists a non-trivial solution $u\neq0$ with the above gaussian decay properties, with $\alpha\beta=4T$.

In \cite{BFGRV, CF}, the authors consider magnetic perturbations and study
the validity of the previous statements. Some geometric restrictions on the magnetic field $B=DA-DA^t$ naturally arises in the problem.

Our theorems in the sequel are in the same style as the above mentioned ones and complete the picture of the examples in Theorems \ref{thm:harmex} and \ref{prop:UMcounterexample}, proving their sharpness. We first introduce the assumptions on the potentials which will be involved in the next statements.

\subsubsection*{\textbf{(HE)} Assumptions on $V$}
  Let $V=V_1+V_2$, with  
  \begin{equation*}
    V_1=V_1(x):\R^n\to\R,
    \qquad
    V_2=V_2(x,t):\R^{n+1}\to\C,
  \end{equation*}
   and assume that
  \begin{align*}
     & \|V_1\|_{L^\infty} <\infty \\ 
     & \sup_{t\in[0,T]}\left\|e^{\frac{T^2|\cdot|^2}{(\alpha t+\beta
          (T-t))^2}}V_2(\cdot,t)\right\|_{L^\infty}
    e^{\sup_{t\in[0,T]}\left\|\Im V_2(\cdot,t)\right\| _{L^\infty}}
    <\infty.
  \end{align*}
\subsubsection*{\textbf{(HM)} Assumptions on $A$}
Let
$A=(A^1(x),\dots,A^n(x))\in \mathcal C^{1,\varepsilon}_{\text{loc}}(\R^n;\R^n)$.
   Denote by $B=B(x)=DA-DA^t$, $B_{jk}=\partial_{x_j}A^k-\partial_{x_k}A^j$ and assume
  that
   \begin{equation*}
           \|x^tB\|_{L^\infty}^2 <\infty.
    \end{equation*}
    Moreover, assume that there exists a unit vector $\xi\in\mathbb S^{n-1}$ such that
  \begin{equation}\label{hypo:xi} 
    \xi^tB(x)\equiv0.
    \end{equation} 
    Notice that condition \eqref{hypo:xi} cannot hold in dimensions 1,2, since the field $B$ is either null (in 1D) or scalar (in 2D). Due to this, in all the results in the sequel, we will need to restrict to the higher dimensions $n\geq3$ if the magnetic field is present. 
 We are now ready to state our main results: let start with the case of a uniform electric field.
    \begin{theorem}
      \label{thm:UEtheorem} 
      Let \mbox{$n\geq 3$}, and let $u\in\mathcal C([0,T];L^2(\R^n))$ be a
  solution to
  \begin{equation*}
    i\partial_t u -\Delta_A u +V(x,t) u + \left(E(t)\cdot x \right) u + k(t) u=0
  \end{equation*} 
  in $\R^n\times[0,T]$, with $V$ as in (HE) and
  $A$ as in (HM),  $E \in \mathcal C([0,T];\R^n)$ and
  $k \in \mathcal C([0,T];\R)$.

  Assume that
  \begin{align*}
    & \norma*{e^{|\cdot|^2/\beta^2}
      u(\cdot,0)}_{L^2} 
      +\norma*{e^{|\cdot|^2/\alpha^2}
      u(\cdot,T)}_{L^2}<\infty,
  \end{align*}
  for some $\alpha,\beta>0$. If $\alpha\beta<4T$ then $u\equiv0$.

In addition, if $A\equiv 0$, the result holds for any $n\geq 1$.
\end{theorem}
    \begin{theorem}
      \label{thm:QHOtheorem} 
  Let \mbox{$n\geq 3$}, and let $u\in\mathcal C([0,T];L^2(\R^n))$ be a
  solution to
  \begin{equation}\label{eq:QHOequation}
    i\partial_tu -\Delta_A u +V(x,t) u +\frac{\omega^2}{4}|x|^2u =0
  \end{equation} 
  in $\R^n\times[0,T]$, with $V$ as in (HE) and
  $A$ as in (HM) , and
  $0<\omega<\pi/2T$.

 Assume that  
  \begin{align*}
    & \norma*{e^{|\cdot|^2/\beta^2}
      u(\cdot,0)}_{L^2} 
      +\norma*{e^{|\cdot|^2/\alpha^2}
      u(\cdot,T)}_{L^2}<\infty,
  \end{align*}
  for some $\alpha,\beta>0$. If $\alpha\beta<4\frac{\sin(\omega T)}{\omega}$ then $u\equiv0$.

In addition, if $A\equiv 0$, the result holds for any $n\geq 1$.
\end{theorem}
It is not surprising that, for a repulsive quadratic potential (which scales as the harmonic oscillator), the following result holds.
    \begin{theorem}
      \label{thm:QHStheorem} 
      Let \mbox{$n\geq 3$}, and let $v\in\mathcal C([0,T];L^2(\R^n))$ be a
  solution to
  \begin{equation*}
    i\partial_t v -\Delta_A v +V(x,t) v -\frac{\nu^2}{4}|x|^2v =0
  \end{equation*} 
  in $\R^n\times[0,T]$, with $V$ as in (HE) and
  $A$ as in (HM) , and
  $0<\nu<1/T$.

  Assume that
  \begin{align*}
    & \norma*{e^{|\cdot|^2/\delta^2}
      v(\cdot,0)}_{L^2} 
      +\norma*{e^{|\cdot|^2/\gamma^2}
      v(\cdot,T)}_{L^2}<\infty,
  \end{align*}
  for some $\gamma,\delta>0$. If $\gamma\delta<4\frac{\sinh{\nu T}}{\nu}$ then $u\equiv0$.

In addition, if $A\equiv 0$, the result holds for any $n\geq1$.
\end{theorem}
\begin{remark}
  Notice that in the cases $\omega=\nu=0$ the tresholds in Theorems \ref{thm:QHOtheorem},
  \ref{thm:QHStheorem} are coherent with the theory already extablished in \cite{CF} and \cite{EKPV3}.
\end{remark}
In the case of a uniform magnetic potential we have an analogous result, in which we are  able to treat only the even dimensional cases.
\begin{theorem}
  	\label{thm:UMtheorem} 
  Let $n\geq 4$ be an even number, and let $u\in\mathcal C([0,T];L^2(\R^n))$ be a
  solution to
  \begin{equation}\label{eq:UMequation}
     i\partial_tu -\Delta_{A+C} u +V(x,t) u  =0
  \end{equation} 
  in $\R^n\times[0,T]$, with $V$ as in (HE),
  $A$ as in (HM), and
   $ C\colon x\in\R^n \mapsto M x/2 \in \R^n$,
  with $M \in \R^{n\times n}$ such that
  \begin{equation}\label{eq:propM}
    M^t=-M, \quad M^tM=b^2 Id,
  \end{equation}
  for $0<b<\pi/2T$.
  
  Assume that 
  \begin{align*}
    & \norma*{e^{|\cdot|^2/\beta^2}
      u(\cdot,0)}_{L^2} 
      +\norma*{e^{|\cdot|^2/\alpha^2}
      u(\cdot,T)}_{L^2}<\infty,
  \end{align*}
  for some $\alpha,\beta>0$.
 If $\alpha\beta<4\frac{\sin{b T}}{b}$ then $u\equiv0$.

In addition, if $A\equiv 0$, the result holds for any $n\geq 1$.
\end{theorem}
\begin{remark}
  Notice that the magnetic field associated to the potential $C$ is the uniform magnetic field $DC-DC^t=M$.
  The proof of Theorem \ref{thm:UMtheorem} strongly relies on the fact that, expanding the Hamiltonian $-\Delta_{A+C}=-\Delta_A + pert.$, we need to reduce matters to the harmonic oscillator case, in the same spirit as in Remark \ref{rem:conne} above. Consequently, we need $M^tM=\lambda Id$ to be a constant multiple of the identity.  On the other hand, there are no anti-symmetric matrices $M \in \R^{n\times n}$ such that $M^tM=Id$ if $n$ is odd, and 
  this is why we are only able to handle even dimensions. Notice also that 
  Example \ref{ex:UM} is covered by Thm.~\ref{thm:UMtheorem}.
\end{remark}
\begin{remark}
  The geometric condition \eqref{hypo:xi} already appeared in \cite{BFGRV, CF}. At the moment, it is still unclear whether this is necessary or not in our results.
\end{remark}

The proofs of the main results rely on suitable pseudoconformal change of variables, which reduce matters to perturbations of free Schr\"odinger evolutions (roughly speaking, permit to get rid of harmonic oscillators, uniform electric and magnetic potentials). 
This was observed by Niederer in \cite{niederer} (see also references
therein, and \cite{boyer}, \cite{bluman80,bluman83}),  in order to 
determine the Maximal Kinetic Invariance group for the Schr\"odinger equation,
 and deeply analyzed by Takagi in \cite{takagi1,takagi2,takagi3} (see also references therein). 
Here we explain such techniques in a slightly more general
form in Section \ref{sec:changeofvariables}.
We remark that the electro-magnetic potentials we are considering in this paper 
are the only examples of potentials that can be handled 
by these methods (see Remark \ref{rmk:solonoi} and \cite{bluman83}).
We finally remark that the bounds on $\omega T, b T$ in Theorems \ref{thm:QHOtheorem}, \ref{thm:UMtheorem} appear as necessary, in order these changes of variables make sense, but seem to be technical assumptions, which quite likely is possible to avoid.

\subsection*{Aknowledgments}
We wish to thank Gianluca Panati for addressing us 
to the reference \cite{thaller} and for interesting discussions.

\section{Preliminary transformations}\label{sec:changeofvariables}
All the statements in this Section
have to be considered formal: in their application
 some care has to be given in determining the necessary regularity and
in how the considered time intervals change.

The following lemmata can be proven easily by direct computation. We just give some details
for the proof of Corollary \ref{cor:LemmaRotante}.
\begin{lemma}
  Let $u$ be a solution to
\begin{equation*}
i \partial_t u -\Delta_A u + V(x,t) \, u + k(t) \, u =0,
\end{equation*}
for $k=k(t)\in \C$ sufficiently regular.
Set 
\begin{equation*}
\varphi(x,t)=\exp\left[ i \int_0^t k(\tau)\,d\tau\right]   u(x,t),	
\end{equation*}
Then $\varphi$ is solution to
\begin{equation*}
i \partial_t \varphi - \Delta_{ A} \varphi 
		 + V (x,t) \,\varphi
		 = 0.  
\end{equation*}
\end{lemma}
\begin{lemma}[Generalized Galilean transformations]
\label{lem:acceleratore}
Let $u$ be a solution to
\begin{equation*}
i \partial_t u -\Delta_A u + V(x,t) \, u + E(t) \cdot x \, u =0,
\end{equation*}
for $E=E(t)\in\R^n$ sufficiently regular.
Set 
\begin{equation}\label{eq:changeaccelerato}
\varphi(x,t)={\exp}
			\Bigl[ 
			i \frac{\dot S(t)}{2} \cdot x 
			+ i \int_0^t \Bigl(\frac{\dot S(\tau)^2}{4}
			                  - E(\tau)\cdot S(\tau) \Bigr)
					\, d\tau
			\Bigr]
		   u(x+S(t),t),	
\end{equation}
for $S=S(t)\in\R^n$ sufficiently regular. 

Then $\varphi$ is solution to
\begin{equation*}
i \partial_t \varphi - \Delta_{\widetilde A} \varphi 
		 +\widetilde V (x,t) \,\varphi
		 + \left(E(t)+\frac{\ddot S(t)}{2}\right)\cdot x\, \varphi
		= 0,
\end{equation*}
with 
\begin{equation*}
\begin{split}
\widetilde A(x,t)& = A(x+S(t),t), \\
\widetilde V (x,t) &= V(x+S(t),t) + \dot S(t) \cdot \widetilde A (x,t).
\end{split}
\end{equation*}
\end{lemma}
\begin{remark}
  It is useful to read the 
  change of variables in \eqref{eq:changeaccelerato} from a  physical point of view:
  we are changing coordinate system, passing from one in rest
  to one   integral with the accelerating particle.
\end{remark}
\begin{lemma}[Comoving frame]
\label{lem:HOcambio}
Let $u$ be a solution to 
\begin{equation*}
i \partial_t -\Delta_A u + V(x,t) \, u + \frac{h(t)}{4}|x|^2 \, u =0,
\end{equation*}
with $h=h(t)\in\R$ sufficiently regular.
Set 
\begin{equation*}
\varphi(x,t)=a^{-n/2}{\exp}
			\left[ 
			-i \frac{\dot a}{4a}  |x|^2 
			\right]
		   u\left(\frac{x}{a},\int_0^t a(\tau)^{-2}d\tau\right),	
\end{equation*}
for $a=a(t)\in\R$ sufficiently regular. 

Then $\varphi$ is solution to
\begin{equation*}
i \partial_t \varphi - \Delta_{\widetilde A} \varphi 
		 +\widetilde V (x,t) \,\varphi
		 + \left(\frac{\widetilde h(t)}{a^4}-\frac{\ddot a}{a}  \right)\frac{|x|^2}{4}\, \varphi
		= 0,
\end{equation*}
with 
\begin{equation*}
\begin{split}
  \widetilde h(t)&=h\left(\int_0^t a(\tau)^{-2}d\tau\right),\quad 
  \widetilde A(x,t) = \frac{1}{a} A\left(\frac{x}{a},\int_0^t a(\tau)^{-2}d\tau\right), \\
  \widetilde V (x,t) &= \frac{1}{a^2} V\left(\frac{x}{a},\int_0^t a(\tau)^{-2}d\tau\right) -  \frac{\dot a}{a} x  \cdot \widetilde A (x,t).
\end{split}
\end{equation*}
\end{lemma}

\begin{lemma}[Larmor or rotating frame]
\label{lem:LemmaRotante}
Let $u$ be a solution to
\begin{equation*}
i \partial_t u -\Delta_A u + V(x,t) \,u=0.
\end{equation*}
Set 
\begin{equation*}
\varphi(x,t)=   u\left(R(t)\,x,g(t)\right),	
\end{equation*}
for $R=R(t) \in \R^{n\times n}$ and $g=g(t)\in\R$ sufficiently regular. 

Then $\varphi$ is solution to
\begin{equation}\label{eq:nuovamagnetico}
\begin{split}
i \partial_t \varphi &- \Delta_{\widetilde A} \varphi +\widetilde V(x,t) \, \varphi 
		- i R^{-1} \dot R x \cdot \nabla \varphi \\
		=& 
[g'\Delta u (Rx,g) - tr(RR^{t} D^2 u(Rx,g)] \\
& -i[g' (div A) (Rx,g) - R_{ki}R^t_{ij}\partial_k A_j(Rx,g)]\varphi \\
& -2i[g'A(Rx,g)-RR^t A(Rx,g)]\cdot \nabla u(Rx,g) \\
& -[g'|A(Rx,g)|^2-|R^{t}A(Rx,g)|^2]\varphi,
\end{split}
\end{equation}
with 
\begin{equation}\label{eq:defnVLemmaRotante}
  \widetilde V(x,t)= g'V(Rx,g), \quad
  \widetilde A(x,t)= R^{t} A\left(Rx,g\right). 
\end{equation}
\end{lemma}
\begin{corollary}\label{cor:LemmaRotante}
  In the assumptions of Lemma \ref{lem:LemmaRotante}, if $RR^{t}\equiv I$ and $g(t)=t$, 
then \eqref{eq:nuovamagnetico} reads
\begin{equation*}
i \partial_t \varphi - \Delta_{\widetilde A} \varphi 
		+\widetilde V(x,t) \, \varphi
		- i R^{t} \dot Rx \cdot \nabla \varphi 
		= 0.
\end{equation*}
and $\varphi$ is also solution to
\begin{equation*}
i \partial_t \varphi - \Delta_{\bar A} \varphi 
			+ \widetilde V(x,t) \, \varphi
		+ \frac{R^t \dot R x}{2}\cdot
                \left(\widetilde A(x,t)+ \bar A(x,t)\right)
		= 0.
\end{equation*}	
with $\widetilde V,\widetilde A$ defined in \eqref{eq:defnVLemmaRotante} and 
\begin{equation*}
\bar A (x,t)= \widetilde A (x,t) - \frac{R^t \dot R}{2}x.
\end{equation*}
\end{corollary}
\begin{proof}[Proof of Corollary \ref{cor:LemmaRotante}]
The proof is the direct computation. It can be useful to remind that,
thanks to \emph{Jacobi's formula}, we have
\begin{equation*}
div(R^t \dot R x)=tr(R^{t}\dot R)=tr(R^{-1}\dot R)=\frac{(\det R)'}{\det R}=0
\end{equation*}
since $\det R(t)\equiv 1$ or $\det R(t)\equiv -1$.
\end{proof}
\begin{remark}
  It is useful to read the 
  change of variables in Lemma \ref{lem:LemmaRotante} and Corollary \ref{cor:LemmaRotante} from a  physical point of view:
   we are changing coordinate system, passing from one in rest
  to a rotating non-inertial one, thus introducing  fictitious forces.
\end{remark}
\begin{remark}\label{rmk:solonoi}
  We remark that the change of variables considered in this Section
  are of the form
  \begin{equation*}
    \tilde t= \tilde t (t), \quad
    \tilde x_j = a(t) R_{jk}(t) x_k + S_j(t),
  \end{equation*}
  where $\tilde t \colon \R \to \R$, $a \colon \R \to \R$, 
  $R \colon \R \to \R^{n\times n}$, $R(t)^t R(t) = Id$. 
  These are the only coordinate transformations under which
  the Schr\"odinger equation  is covariant. 
  For further 
  details, we remand to Section 5 in \cite{takagi1}.

We also
  remark that the changes of variables in this section can be
  composed, hence general
  electric potentials 
  $V(x,t)=a(t)\abs{x}^2 + b(t)\cdot x + c(t)$,
  for $a,c\colon \R \to R$, $b \colon \R \to \R^n$
  and linear magnetic potentials $A=M(t)x$ for 
  $M\colon \R \to \R^{n\times n}$ can be considered.
\end{remark}


\section{Proof of Theorem \ref{thm:UEtheorem}}
We perform the following change of variables, immediate consequence of
Lemma \ref{lem:acceleratore}: the proof is omitted.
\begin{proposition}
  Let $u \in \mathcal C([0,T];L^2(\R^n))$ be a solution to
  \begin{equation*}
    i \partial_t u -\Delta_A u +V(x,t) u +\left(E(t)\cdot x\right) u + k(t) u =0,
  \end{equation*}
  with $A=A(x,t):\R^n\times\R \to \R^n$, $V=V(x,t):\R^n\times\R \to \C$, $E \in C([0,T];\R^n)$,
  $k \in \mathcal C([0,T];\R)$, 
  and set 
\begin{equation*}
\varphi(x,t)={\exp}
			\Bigl[ i\int_0^t k(\tau)\,d\tau +
			i \frac{\dot S(t)}{2} \cdot x 
			+ i \int_0^t \Bigl(\frac{\dot S(\tau)^2}{4}
			                  - E(\tau)\cdot S(\tau) \Bigr)
					\, d\tau
			\Bigr]
		   u(x+S(t),t),	
\end{equation*}
for $S\in\mathcal C^2([0,T];\R^n)$ defined as
\begin{equation*}
  S(t):=-2\left(\int_0^t \int_0^s E(\tau)\,d\tau ds-
    \frac{t}{T}
    \int_0^T \int_0^s E(\tau)\,d\tau ds 
  \right).
\end{equation*}

Then $\varphi \in \mathcal C([0,T];L^2(\R^n))$ and it is solution to
  \begin{equation*}
    i\partial_t \varphi - \Delta_{\widetilde A} \varphi + \widetilde V \varphi = 0,
  \end{equation*}
  where $\widetilde A$ and $\widetilde V$ verify 
(HM) and (HE)
respectively, and
\begin{equation*}
\begin{split}
\widetilde A(x,t)& = A(x+S(t),t), \\
\widetilde V (x,t) &= V(x+S(t),t) + \dot S(t) \cdot \widetilde A (x,t).
\end{split}
\end{equation*}
\end{proposition}
We observe that
\begin{equation*}
  \begin{split}
    & \norma*{e^{|\cdot|^2/\beta^2}u(\cdot,0)}_{L^2}
    = \norma*{e^{|\cdot|^2/\beta^2}\varphi(\cdot,0)}_{L^2} \\
    & \norma*{e^{|\cdot|^2/\alpha^2}u(\cdot,T)}_{L^2}
    = \norma*{e^{|\cdot|^2/\alpha^2}\varphi(\cdot,T)}_{L^2}.
  \end{split}
\end{equation*}
We conclude thanks to Theorem 1.3 in \cite{CF} for the case $n\geq 3$ and $A\neq 0$, and Theorem 1 in \cite{EKPV3}
for the case $n\geq1$ and $A\equiv 0$.

\section{Proofs of Theorems \ref{thm:QHOtheorem} and \ref{thm:QHStheorem}}
The two proofs largely overlap, hence they will be given together.
The proofs are divided into three steps.
  
\subsection{Cr\"onstrom gauge}\label{sec:cronstrom}
The first step consists in reducing to the Cr\"onostrom gauge 
\begin{equation*}
  x\cdot A(x)=0 \quad \text{ for all }x \in  \Rn,
\end{equation*} 
by means of the following result.

\begin{lemma}\label{lem:cronstrom1}
  Let $A=A(x)=(A^1(x),\dots,A^n(x)):\R^n\to\R^n$, for $n\geq2$
 and denote by $B=DA-DA^t\in \mathcal M_{n\times
n}(\R)$, $B_{jk}=A^k_j-A^j_k$, and $\Psi(x):=x^tB(x)\in\R^n$. 
Assume that the two vector quantities
  \begin{equation}\label{eq:cronstrom1}
    \int_0^1A(sx)\,ds\in\R^n,
    \qquad
    \int_0^1\Psi(sx)\,ds\in\R^n
  \end{equation}
  are finite, for almost every $x\in\R^n$; moreover, define the (scalar) function
  \begin{equation*}
    \varphi(x):=x\cdot\int_0^1A(sx)\,ds\in\R.
  \end{equation*}
  Then, the following two identities hold:
  \begin{align*}
    &\widetilde A(x):=A(x)-\nabla\varphi(x)  = -\int_0^1\Psi(sx)\,ds
    \\
    &x^tD\widetilde A(x)   = -\Psi(x) +\int_0^1\Psi(sx)\,ds.
  \end{align*}
\end{lemma}
\begin{remark}\label{rem:cronstrom}
Notice that
  \begin{equation}\label{eq:gaugefinal}
    x\cdot\widetilde A(x) \equiv0,
    \qquad
    x\cdot x^tD\widetilde A(x)\equiv0.
  \end{equation}
  From now on, we will hence assume, without loss of generality, 
  that \eqref{eq:gaugefinal} are satisfied by $A$.
  Observe moreover that assumption \eqref{hypo:xi} in Theorem \ref{thm:QHOtheorem} 
  is preserved by the above gauge transformation, and we have in addition that $A\cdot\xi\equiv0$. 
We also remark that
\begin{equation*} 
\norma{\tilde A}_{L^\infty}^2 +   \norma{x^t B}_{L^\infty}^2 < +\infty. 
\end{equation*}
Finally notice that the first condition in \eqref{eq:cronstrom1} is guaranteed
by the assumption $A\in \mathcal C^{1,\varepsilon}_{\text{loc}}$ in {(HM)} in the Introduction.

We mention \cite{I} for the proof of the previous Lemma; see alternatively Lemma 2.2 in \cite{BFGRV}.
\end{remark}

\subsection{Removing the harmonic oscillator}
The second step consists in reducing the proof to the case of the equation without 
harmonic oscillator, by means of the appropriate change of variables, as 
shown in the following propositions, immediate consequences of 
Lemma \ref{lem:HOcambio}. We omit the proofs. 
\begin{proposition}
  Let $u \in \mathcal C([0,T];L^2(\R^n))$ be a solution to
  \begin{equation*}
    i \partial_t u -\Delta_A u +V(x,t) u +\frac{\omega^2}{4}|x|^2 u =0,
  \end{equation*}
  with $A=A(x,t):\R^n\times\R \to \R^n$, $V=V(x,t):\R^n\times\R \to \C$, $0<\omega<\pi/2T$,
  and set 
  \begin{equation*}
    \varphi(x,t):= (1+\omega^2 t^2)^{-\frac{n}{4}} 
                   \exp\left[-\frac{i\omega^2 t}{4+4\omega^2 t^2}|x|^2\right]
         	   u\left( \frac{x}{\sqrt{1+\omega^2 t^2}}, \frac{\arctan{\omega t}}{\omega} 
		   \right).
  \end{equation*}
  Then $\varphi \in \mathcal C([0,\tan(\omega T)/\omega];L^2(\R^n))$ and it is solution to
  \begin{equation*}
    i\partial_t \varphi - \Delta_{\widetilde A} \varphi + \widetilde V \varphi = 0,
  \end{equation*}
  where 
  \begin{align}
      \widetilde A(x,t)& = \frac{1}{\sqrt{1+\omega^2 t^2}} 
      A\left( \frac{x}{\sqrt{1+\omega^2 t^2}}, \frac{\arctan{\omega t}}{\omega} \right),\notag \\
      \widetilde V (x,t) &= \frac{1}{1+\omega^2 t^2} 
      V\left( \frac{x}{\sqrt{1+\omega^2 t^2}}, \frac{\arctan{\omega t}}{\omega}\right) 
      -  \frac{\omega^2 t}{1+\omega^2 t^2} x  \cdot \widetilde A (x,t). \label{eq:defnVPositive}
   \end{align}

  Moreover for all $t \in [0,T]$
  \begin{equation*}
    u(x,t)=(\cos{\omega t})^{-\frac{n}{2}}
    e^{i\tan{\omega t}\frac{\omega |x|^2 }{4}}
  \varphi\left(
  \frac{x}{\cos{\omega t}}, \frac{\tan{\omega t}}{\omega}
  \right).
\end{equation*}
\end{proposition}
\begin{proposition}
  Let $v \in \mathcal C([0,T];L^2(\R^n))$ be a solution to
  \begin{equation*}
    i \partial_t v -\Delta_A v +V(x,t) v -\frac{\nu^2}{4}|x|^2 v =0,
  \end{equation*}
  with $A=A(x,t):\R^n\times\R \to \R^n$, $V=V(x,t):\R^n\times\R \to \C$, $0<\nu<1/T$,
  and set 
  \begin{equation*}
    \psi(x,t):= (1-\nu^2 t^2)^{-\frac{n}{4}} 
                   \exp\left[\frac{i\nu^2 t}{4-4\nu^2 t^2}|x|^2\right]
		   v\left( \frac{x}{\sqrt{1-\nu^2 t^2}}, \frac{\tanh^{-1}{\nu t}}{\nu} 
		   \right).
  \end{equation*}
  Then $\psi \in \mathcal C([0,\tanh(\nu T)/\nu];L^2(\R^n))$ and it is solution to
  \begin{equation*}
    i\partial_t \psi - \Delta_{\widetilde A} \psi + \widetilde V \psi = 0,
  \end{equation*}
  where 
  \begin{align}
      \widetilde A(x,t)& = \frac{1}{\sqrt{1-\nu^2 t^2}} 
      A\left( \frac{x}{\sqrt{1-\nu^2 t^2}}, \frac{\tanh^{-1}{\nu t}}{\nu} \right),\notag \\
      \widetilde V (x,t) &= \frac{1}{1-\nu^2 t^2} 
      V\left( \frac{x}{\sqrt{1-\nu^2 t^2}}, \frac{\tanh^{-1}{\nu t}}{\nu}\right) 
	+  \frac{\nu^2 t}{1-\nu^2 t^2} x  \cdot \widetilde A (x,t). \label{eq:defnVNegative}
   \end{align}

  Moreover for all $t \in [0,T]$
  \begin{equation*}
    v(x,t)=(\cosh{\nu t})^{-\frac{n}{2}}
    e^{i\tanh{\nu t}\frac{\nu |x|^2 }{4}}
  \psi\left(
  \frac{x}{\cosh{\nu t}}, \frac{\tanh{\nu t}}{\nu}
  \right).
\end{equation*}
\end{proposition}

We remark that the second term at right hand side of \eqref{eq:defnVPositive} and 
\eqref{eq:defnVNegative} 
vanishes, thanks to \eqref{eq:gaugefinal}.
Moreover, the function $\widetilde V$  verifies assumptions (HE) 
  and $\widetilde A$ verifies assumptions (HM) in the Introduction.

\subsection{Conclusion of the proof}
We remark that
\begin{equation*}
  \begin{split}
    & \norma*{e^{|\cdot|^2/\beta^2}u(\cdot,0)}_{L^2}
    = 
    \norma*{e^{|\cdot|^2/\beta^2}\varphi(\cdot,0)}_{L^2} \\
    & \norma*{e^{|\cdot|^2/\alpha^2}u(\cdot,T)}_{L^2}
    =
    \norma*{e^{\cos^2(\omega T)|\cdot|^2/\alpha^2}\varphi\left(\cdot,\frac{\tan{\omega T}}{\omega}\right)}_{L^2}. \\
   & \norma*{e^{|\cdot|^2/\delta^2}v(\cdot,0)}_{L^2}
    = 
    \norma*{e^{|\cdot|^2/\delta^2}\psi(\cdot,0)}_{L^2} \\
    & \norma*{e^{|\cdot|^2/\gamma^2}v(\cdot,T)}_{L^2}
    =
    \norma*{e^{\cosh^2(\nu T)|\cdot|^2/\gamma^2}\psi\left(\cdot,\frac{\tanh{\nu T}}{\nu}\right)}_{L^2}.
    \end{split}
\end{equation*}
We conclude thanks to Theorem 1.3 in \cite{CF} for the case $n\geq 3$ and $A\neq 0$, and Theorem 1 in \cite{EKPV3}
for the case $n\geq 1$ and $A\equiv 0$.


\section{Proof of Theorem \ref{thm:UMtheorem}}

Analogously to what we have done in Section \ref{sec:cronstrom}, 
we reduce the problem to the Cr\"onstrom gauge, thanks to Lemma \ref{lem:cronstrom1}.
We remark that (with the notations of Lemma \ref{lem:cronstrom1})
\begin{equation*}
  \widetilde{A+C}=\widetilde A + C,
\end{equation*}
since $M$ is an anti-symmetric matrix. The Remark \ref{rem:cronstrom} is valid and
 we will omit the tildes in the following.

By means of the appropriate change of variables, we can suppress the magnetic potential $C$ in 
\eqref{eq:UMequation}.
\begin{proposition}
  Let $u \in \mathcal C([0,T];L^2(\R^n))$, be a solution to
  \begin{equation*}
    i \partial_t u -\Delta_{A+C} u +V(x,t) u =0,
  \end{equation*}
  with $A=A(x,t):\R^n\times\R \to \R^n$, $V=V(x,t):\R^n\times\R \to \C$, $C=M x/2$, with
  $M\in \R^{n\times n}$, $M^t=-M$. Set 
  \begin{equation*}
    \varphi(x,t):= u\left( e^{Mt}x, t
		   \right).
  \end{equation*}
  Then $\varphi \in \mathcal C([0,T];L^2(\R^n))$ and it is solution to
  \begin{equation*}
    i\partial_t \varphi - \Delta_{\widetilde A} \varphi + \widetilde V \varphi + \frac{|Mx|^2}{4}\varphi= 0,
  \end{equation*}
  where 
  \begin{equation*}
    \widetilde A(x,t) = e^{-Mt} A\left( e^{Mt}x,t \right), 
    \qquad
    \widetilde V (x,t) = V\left( e^{Mt}x, t\right).
   \end{equation*}
\end{proposition}
\begin{proof}
  The proposition is a simple application of Lemma \ref{lem:LemmaRotante}
  and Corollary \ref{cor:LemmaRotante}: we choose $R \in C([0,T],\R^{n \times n})$ such that 
  $\dot R(t)= MR(t)$.
\end{proof}

We conclude the proof observing that, thanks to \eqref{eq:propM}, 
\begin{equation*}
  \abs{Mx}^2 = (M^tMx,x) = b^2 |x|^2,
\end{equation*}
whence the thesis, thanks to Theorem \ref{thm:QHOtheorem}.

\end{document}